\UseRawInputEncoding
\documentclass[11pt,reqno]{amsart}
\usepackage{amsmath, amssymb, amsthm, esint, verbatim, hyperref}
\usepackage{times}
\usepackage{dsfont}
\usepackage{accents}

\usepackage{wrapfig}
\usepackage{tikz}
\usetikzlibrary{decorations.fractals}

\numberwithin{equation}{section}

\hypersetup{
  pdftitle={Regularized distance in space form and its applications},
  pdfauthor={Xiaodong Cao and Ling Xiao},
  pdfsubject={},
  pdfkeywords={},
  pdfpagelayout=SinglePage,
  pdfpagemode=UseOutlines,
  colorlinks,
  bookmarksopen,
  linkcolor=[rgb]{0,0,0.7},
  urlcolor=[rgb]{0,0,0.4},
  citecolor=[rgb]{0.4,0.1,0},
}

\newtheorem{theorem}{Theorem}[section]

\newtheorem{proposition}[theorem]{Proposition}

\newtheorem{lemma}[theorem]{Lemma}

\theoremstyle{definition}
\newtheorem{definition}[theorem]{Definition}
\theoremstyle{remark}
\newtheorem{remark}[theorem]{Remark}
\theoremstyle{remark}

\theoremstyle{remark}

\theoremstyle{remark}

\theoremstyle{remark}

\newcommand{\be}{\begin{equation}}
\newcommand{\ee}{\end{equation}}
\newcommand{\ju}[2]{\begin{array}{#1}#2\end{array}}

\newcommand{\ubar}[1]{\underaccent{\bar}{#1}}
\newcommand*\Laplace{\mathop{}\!\mathbin\bigtriangleup}
\newcommand{\ssubset}{\subset\joinrel\subset}
\newcommand{\td}{\tilde}

\newcommand{\lt}{\left}
\newcommand{\rt}{\right}

\newcommand{\goto}{\rightarrow}
\newcommand{\R}{\mathbb{R}}

\newcommand{\e}{\epsilon}
\newcommand{\s}{\sigma}
\newcommand{\p}{\partial}
\newcommand{\al}{\alpha}

\newcommand{\fb}{\mathfrak b}
\newcommand{\fbo}{\mathfrak{b}^1}
\newcommand{\fbz}{\mathfrak{b}^0}
\newcommand{\N}{\mathbb N}
\newcommand{\mH}{\mathbb H}
\newcommand{\mS}{\mathbb S}
\newcommand{\bn}{\bar\nabla}
\newcommand{\pr}{\phi(r)}
\newcommand{\prt}{\phi^2(r)}
\newcommand{\dpr}{\dot\phi(r)}
\newcommand{\pro}{\phi(\rho)}
\newcommand{\prot}{\phi^2(\rho)}
\newcommand{\dpro}{\dot\phi(\rho)}
\newcommand{\nb}{\nabla}
\newcommand{\lu}{\ubar{u}}
\newcommand{\la}{\lambda}
\newcommand{\lul}{\ubar{u}_{\text{loc}}}
\newcommand{\lbg}{\mathbf{\ubar{g}}}
\newcommand{\lbf}{\mathbf{\ubar{f}}}
\newcommand{\bbg}{\mathbf{\bar g}}
\newcommand{\bbf}{\mathbf{\overline f}}

\title{Regularized distance in space forms and its application}
\author{Xiaodong Cao and Ling Xiao}


\address{Department of Mathematics, Cornell University, Ithaca, NY 14853}
\email{xiaodongcao@cornell.edu}

\address{Department of Mathematics, University of Connecticut, Storrs, CT 06269}
\email{ling.2.xiao@uconn.edu}

\begin{document}
\maketitle
\begin{center}
Dedicated to Yanyan Li on his 65th birthday with admiration and friendship.
\end{center}

\begin{abstract}
We construct regularized distance functions for star-shaped domains in space forms $\N^n(K)$ and derive explicit formulas for their Hessians in terms of the principal curvatures of the boundary. As an application, we use these regularized distance functions to construct barriers and prove the existence of an admissible $C^{1,1}$ solution to the degenerate  $\sigma_k$ equation on ring domains in $\R^n$ and $\mH^n$, under suitable conditions on the boundary components.
\end{abstract}
\maketitle
\tableofcontents

\section{Introduction}
It is well understood that the distance function is a powerful tool for constructing barriers in the study of Dirichlet problems. However, for most domains $\Omega,$ the distance function to the boundary $\p\Omega$ is only well defined in a small neighborhood of $\p\Omega.$ This limitation creates difficulties in barrier constructions, especially when $\Omega$ has a genus greater than or equal to one.
Therefore, it is sometimes more convenient to use a regularized distance, as described below, to construct barriers.

We note that the idea of finding replacements for the distance function to construct barriers appeared many years ago, in specialized circumstances. For example, a regularized distance for arbitrary domains was constructed by Triebel \cite{Tri78}, while a construction for Lipschitz domains was given by Ne\v cas \cite{Nec62}. An alternative construction of a regularized distance appears in \cite{Lie85}, where the properties of the regularized distance are also studied in detail. Inspired by these earlier works, the second-named author constructed a regularized distance for star-shaped domain in $\R^n$ and studied its Hessian (see \cite{Xiao22}). As an application, a sharp generalized Minkowski inequality was proved in \cite{Xiao22} for smooth, $(k-1)$-convex, star-shaped domains. Another application of this regularized distance can be found in \cite{LX25}, where the existence of a solution to the exterior Dirichlet problem for Hessian equations on a non-convex ring was established.

Our primary objectives are to further develop the ideas introduced in
\cite{Xiao22} to construct regularized distance functions in space forms and to investigate their Hessian properties.

We denote by $\N^{n}(K)$ a space form of sectional curvature $K=-1,0$, or $+1,$ and denote by $g^{\N}:=ds^{2}$ the Riemannian metric of $\N^{n}(K)$. In $\mathbb{R}^{n}$, let $\mS^{n-1}$ be the unit sphere centered at the origin. Suppose that $(z, r)$ are spherical coordinates in $\mathbb{R}^{n},$ where $z \in \mS^{n-1}$. The standard metric on $\mS^{n-1}$ induced from $\mathbb{R}^{n}$ is denoted by $dz^{2}$. Let $0<a\leq\infty$, $I = [0, a),$ and $\pro$ is a
nonnegative function on $I.$ Then
$$
g^{\N}:=ds^{2}=dr^{2}+\phi^{2}(r)dz^{2}
$$
is a model of $\N^{n}(K)$. More precisely, $\N^{n}(K)$ is the Euclidean space $\mathbb{R}^{n}$ if $\phi(r)=r,\,\, a=\infty$; $\N^{n}(K)$ is the unit $n$-sphere $\mS^{n}$ with constant sectional curvature $+1$ if $\phi(r)=\sin r,\,\,a=\pi$; and $\N^{n}(K)$ is the hyperbolic space $\mH^{n}$ with constant sectional curvature $-1$ if
$\phi(r)=\sinh r,\,\, a=\infty.$

Before stating our main results, we first define regularized distance for a domain $\Omega$. In this paper, we define the \it{\textbf{signed distance}} to
$\p\Omega$ by
\[
d(x)=\left\{\begin{aligned}
&-\text{dist}(x, \p\Omega)\,\,x\in\Omega\\
&\text{dist}(x, \p\Omega)\,\, x\notin\Omega.
\end{aligned}
\right.
\]

\begin{definition}
\label{def-reg-dist}
Let $\Omega\subset\N^n(K)$ be a smooth, compact domain in space forms, containing the origin $\{0\}$ as an interior point and with nonempty boundary $\p\Omega.$ A function $\vartheta$ is called a {\it\textbf{regularized distance}} for $\Omega$
if $\vartheta\in C^2(\mathbb \N^n(K)\setminus\{0\})\cup C^{0, 1}(\N^n(K))$
and if the ratios $\vartheta(x)/d(x)$ and $d(x)/\vartheta(x)$ are positive and uniformly bounded from above for all $x\in\N^n(K)\setminus\p\Omega.$
Here, $d(x)$ is the signed distance to $\p\Omega.$
\end{definition}
We note that although our definition of the regularized distance differs from the one in \cite{Lie85}, the two share an important common property: both are proportional to the signed distance to $\p\Omega.$

We prove the following theorems.
\begin{theorem}
\label{thm1}
Let $\Omega\subset\N^n(K)$ be a smooth, star-shaped domain with
respect to the origin in the space forms $\N^n(K)$. That is, $\Gamma:=\p\Omega$ can be parameterized as a
graph of the radial function $\rho(z): \mS^{n-1}\goto\R,$ i.e.,
\[\Gamma=\{(z, \rho(z)): z\in\mS^{n-1}\}.\] Then, there is a regularized distance for $\Omega$ defined by
\[\fb(z, r): =\frac{\phi(r-\rho(z))}{\phi(\rho(z))}\,\,\mbox{for}\,\,(z, r)\in\N^n(K).\]
Moreover, for any $(z, r)\in\N^n(K)\setminus\{0\},$ under a specifically chosen coordinate system
\be\label{hessian-thm-b}
\begin{aligned}
\text{Hessian}(\fb)&=\left[\ju{ccccc}{\frac{w^3}{\phi(r)}a_{11}&\frac{w^2}{\phi(r)}a_{12}&\cdots&\frac{w^2}{\phi(r)}a_{1n-1}&0\\
\frac{w^2}{\phi(r)}a_{12}&\frac{w}{\phi(r)}a_{22}&\cdots&0&0\\
\vdots&\vdots&\ddots&\vdots\\
\frac{w^2}{\phi(r)}a_{1n-1}&0&\cdots&\frac{w}{\phi(r)}a_{n-1n-1}&0\\
0&0&\cdots&0&0}\right]-K\fb I_n.
\end{aligned}
\ee
where $I_n$ is the $n\times n$ identity matrix, $w=\sqrt{1+\frac{|\nabla'\rho|^2}{\phi^2(\rho)}},$ and the eigenvalues of $(a_{ij}),$ namely $\kappa[a_{ij}]=(\kappa_1, \cdots, \kappa_{n-1}),$ are the principal curvatures of $\Gamma$ at $(z, \rho(z)).$
\end{theorem}

A domain $\mathcal M\subset\N^n(K)$ is called a {\it\textbf{ring}} if $\mathcal M=\Omega_0\setminus\bar\Omega_1,$ where $\Omega_0$ and $\Omega_1$ are two
bounded domains in $\N^n(K)$ such that $\Omega_1\ssubset\Omega_0.$
As an application of Theorem \ref{thm1}, when $\N^n(K)=\R^n$ and $\mH^n,$ we consider the following degenerate $\s_k$ equation defined on a ring.
\be\label{int-main}
\left\{\begin{aligned}
F\lt(\bn^2u, u\rt):=\s_k^{\frac{1}{k}}\lt(\lambda\lt(\bn^2u+Kug^\N\rt)\rt)&=0\,\,&\mbox{in}\,\,\mathcal{M}\\
u&=\frac{K+1}{2}\,\,&\mbox{on}\,\,\p\Omega_0\\
u&=\frac{K-1}{2}\,\,&\mbox{on}\,\,\p\Omega_1,
\end{aligned}\right.
\ee
where $$\s_k(\lambda)=\sum\limits_{1\leq i_1<\cdots<i_k\leq n}\la_{i_1}\cdots\la_{i_k}$$ is the $k$-th  elementary symmetric function of $\la\in\R^n,$ $\bn^2 u$ is the Hessian of $u$ in $\N^n(K),$ and for a $(0, 2)$ tensor $h$ in $\N^n(K),$
$\la(h)=(\la_1, \cdots, \la_n)$ denotes the eigenvalues of $h$ with respect to the metric $g^\N.$

We denote $\Sigma_k$ as the G{\aa}rding's cone
\[\Sigma_k=\{ \la\in \R^{n}: \s_m(\la) > 0, m = 1,\cdots , k\}.\]
For any open set $U\subset\N^n(K)$, a function $v\in C^{2}(U)$ is said to be {\it\textbf{ admissible}} if $\la(\bn^2v+Kvg^\N)\in\bar\Sigma_{k}.$ A $C^2$ regular hypersurface $\Gamma\subset\N^{n}$ is called \textbf{strictly k-convex} ($k$-convex) if its principal curvature vector
$\kappa(X)\in\Sigma_k$ ($\kappa(X)\in\bar\Sigma_k$) for all $X\in\Gamma.$ We say a domain $\Omega$ is strictly $k$-convex ($k$-convex) if $\partial\Omega$ is strictly $k$-convex ($k$-convex).

In this paper, we assume that $\p\Omega_i,$ $i=0, 1,$ is a graph of a radial function $\rho^i(z): \mS^{n-1}\goto\R,$ that is,
$\p\Omega_i=\{(z, \rho^i(z)): z\in\mS^{n-1}\}.$ We also adopt the convention that for an $n-1$ dimensional hypersurface $\p\Omega\subset\N^n,$ strict $(n-1)$-convexity and strict $n$-convexity are equivalent.
In particular, under this convention, the condition $\kappa(\p\Omega)\in\Sigma_{n}$ simply means $\p\Omega$ is strictly convex. We prove
\begin{theorem}
\label{thm2}
For any $2\leq k\leq n$ let $\mathcal M$ be a domain that satisfies
\begin{itemize}
\item when $\N^n=\R^n,$ $\p\Omega_i,$ $i=0, 1,$ are strictly $(k-1)$-convex,
\item when $\N^n=\mH^n,$ $\p\Omega_i,$ $i=0, 1,$ are strictly $k$-convex and $\max\limits_{z\in\mS^{n-1}}(\rho^0(z)-\rho^1(z))<\tanh^{-1}\lt(\frac{1}{\sqrt{n-2}}\rt)$,
\end{itemize}
Then there is an admissible solution $u\in C^{1, 1}(\bar{\mathcal M})$ of \eqref{int-main} with
\be\label{app-thm-c2}
\|u\|_{C^{1, 1}(\mathcal M)}\leq C.
\ee
\end{theorem}

\section{Hessian in space forms}
\label{sec-hess}
In this section, we derive expressions for the Hessian in space forms.
In the following, we will denote the standard connection in $\N^n$ by $\bn.$
Now, we choose a local orthonormal frame $\{e_1, \cdots, e_{n-1}\}$ on the unit sphere $\mathbb{S}^{n-1}$. Let $\tau_{a}:=\dfrac{e_{a}}{\phi(r)}$, $1\leq a\leq n-1,$ which is the orthonormal frame on the sphere with radius $\phi(r),$ and we also let $\tau_r:=\frac{\p}{\p r}$.
Then we have
\be\label{hess1}
\bn f=f_r\tau_r+\sum\limits_{a=1}^{n-1}\frac{1}{\phi(r)}f_{a}\tau_{a},
\ee
where $f_r=\tau_r f, f_{a}=e_{a}f.$ Moreover,
\[
\begin{aligned}
\bn^2f&=\bn_{\tau_r}\lt(f_r\tau_r+\sum\limits_{a=1}^{n-1}\frac{1}{\phi(r)}f_{a}\tau_{a}\rt)\otimes \tau_r\\
&+\sum\limits_{b=1}^{n-1}\frac{1}{\phi(r)}\bn_{e_{b}}\lt(f_r\tau_r+\sum\limits_{a=1}^{n-1}\frac{1}{\phi(r)}f_{a}\tau_{a}\rt)\otimes \tau_{b}.
\end{aligned}
\]
It is well known that $V=\phi(r)\frac{\p}{\p r}$ is a conformal Killing vector field (see \cite{GL15} Lemma 2.1).
This implies
\[\bn_{\tau_a}V=\dot\phi(r)\tau_a,\,\,1\leq a\leq n-1,\]
and
\[\bn_{\tau_r}V=\dot\phi(r)\tau_r,\]
where $\dot\phi(r):=\frac{d\phi(r)}{dr}.$
Therefore, for $1\leq a, b\leq n-1$ we have
\[ \bn_{\tau_{a}}\tau_r=\frac{\dot\phi(r)}{\phi(r)}\tau_{a},\,\,\bn_{\tau_r}\tau_r=0.\]
Moreover, by $\bn_{\tau_b}\lt<\tau_a, \tau_r\rt>=0$ we obtain
$$\bn_{\tau_{b}}\tau_{a}=-\frac{\dot\phi(r)}{\phi(r)}\tau_r\delta_{ab}.$$
Finally, we compute
\[\bn_{\tau_r}\tau_a=\bn_{\tau_r}\lt(\frac{1}{\pr}\rt)e_a+\frac{1}{\pr}\bn_{\tau_r}e_a
=-\frac{\dpr}{\prt}e_a+\bn_{\tau_a}\tau_r=0.\]

Thus, the Hessian of $f$ in the space form is
\be\label{hess1.1}\bn^2_{ab}f=\bn^2f(\tau_{a},\tau_{b})=\frac{1}{\prt}f_{ab}+\frac{\dpr}{\pr}f_r\delta_{ab},\ee
\be\label{hess1.2}\bn^2_{a r}f=\bn^2f(\tau_{a},\tau_r)=\frac{1}{\pr}f_{a r}-\frac{\dpr}{\prt}f_{a},\ee
and
\be\label{hess1.3}\bn^2_{rr}f=\bn^2f(\tau_r,\tau_r)=f_{rr}.\ee
Here $1\leq a, b\leq n-1,$ $f_{ab}=e_{b}e_{a}f, f_{a r}=\tau_re_{a} f,$ and $f_{rr}=\tau_r\tau_rf$.

\section {Regularized distance function}
\label{sec-rdf}

\subsubsection{Radial graph in $\N^n$}
\label{sub-rdf.1}
Let $(\Gamma, g)$ be a hypersurface in $\N^n$ with induced metric $g.$ In the following, we shall give the local expression of induced metric
and second fundamental form of $\Gamma$ when
\[\Gamma:=\{(z, \rho(z)): z\in\mS^{n-1}\}\]
is a radial graph in $\N^n.$ Let $z^1, \cdots, z^{n-1}$ be smooth local coordinates in a coordinate neighborhood $U\subset\mS^{n-1},$
then $e_i:=\frac{\p}{\p z^i},\,\,i=1, \cdots, n-1,$ are the corresponding local frame of tangent vectors on $U.$
We shall denote by $e$ the standard spherical metric, $\nb'$ the covariant derivative with respect to $e,$ and $e_{ij}:=e(e_i, e_j).$
Then the first and second fundamental form of $\Gamma$ are (see \cite{BLO02})
\[g_{ij}=\prot e_{ij}+\rho_i\rho_j,\,\,
g^{ij}=\frac{1}{\prot}\lt(e^{ij}-\frac{\rho^i\rho^j}{\prot+|\nabla'\rho|^2}\rt),
\]
and
\[h_{ij}=\frac{\pro}{\sqrt{\prot+|\nb'\rho|^2}}
\lt(-\nb'_{ij}\rho+\frac{2\dpro}{\pro}\rho_i\rho_j+\pro\dpro e_{ij}\rt),\]
where $\rho_i=e_i\rho,$ $e^{ij}=(e_{ij})^{-1},$ and $\rho^i=e^{ij}\rho_j.$

Now, at any fixed point $z_0\in\mS^{n-1},$ let $\{z^1, \cdots, z^{n-1}\}$ be the normal coordinates at $z_0.$ Then the principal curvatures of
$\Gamma$ at $(z_0, \rho(z_0))$ are the eigenvalues of the symmetric matrix $(a_{ij})$ for
\[a_{ij}:=\gamma^{ik}h_{kl}\gamma^{lj}.\]
Here, $\gamma^{ij}:=\frac{1}{\pro}\lt(\delta_{ij}-\frac{\rho_i\rho_j}{\phi^2(\rho)w(1+w)}\rt)$ and $w:=\sqrt{1+\frac{|\nb'\rho|^2}{\prot}}.$ We note that
$\gamma^{ij}$ is the square root of $g^{ij},$ i.e., $\sum\limits_k\gamma^{ik}\gamma^{kj}=g^{ij}.$ To simplify the expression of $a_{ij}$ at $z_0,$ we may rotate the coordinates
such that $|\nb'\rho(z_0)|=\rho_1(z_0)$ and
\[\nb'_{\alpha\beta}\rho(z_0)=\frac{\p^2\rho(z_0)}{\p z^\alpha\partial z^\beta}=:\rho_{\al\al}(z_0)\delta_{\al\beta}
\,\,\mbox{for}\,\,2\leq\al, \beta\leq n-1.\] Then at the point $(z_0, \rho(z_0))\in \Gamma,$ under the normal coordinates chosen above, we have \[
\left\{
\begin{aligned}
\gamma^{11}&=\frac{1}{\pro}\lt(1-\frac{w^2-1}{w(1+w)}\rt)=\frac{1}{\pro w},\\
\gamma^{1\al}&=0,\,\,&2\leq\al\leq n-1,\\
\gamma^{\al\beta}&=\frac{1}{\pro}\delta_{\al\beta},\,\,&2\leq\al, \beta\leq n-1,
\end{aligned}
\right.
\]
\[\left\{
\begin{aligned}
h_{11}&=\frac{1}{w}\lt(-\rho_{11}+\frac{2\dpro}{\pro}\rho_1^2+\pro\dpro\rt),\\
h_{1\al}&=-\frac{1}{w}\rho_{1\al},\,\,&2\leq\al\leq n-1,\\
h_{\al\beta}&=\frac{1}{w}\lt(-\rho_{\al\al}+\pro\dpro\rt)\delta_{\al\beta},\,\,&2\leq\al, \beta\leq n-1,
\end{aligned}
\right.
\]
and
\be\label{second-fundamental-form}
\left\{
\begin{aligned}
a_{11}&=\gamma^{1k}h_{kl}\gamma^{l1}=\gamma^{11}h_{11}\gamma^{11}=\frac{h_{11}}{\prot w^2},\\
a_{1\al}&=\gamma^{1k}h_{kl}\gamma^{l\alpha}=\frac{h_{1\al}}{\prot w},\,\,&2\leq\al\leq n-1,\\
a_{\al\beta}&=\gamma^{\al\al}h_{\al\beta}\gamma^{\beta\beta}=\frac{h_{\al\beta}}{\prot},\,\,&2\leq\al, \beta\leq n-1.
\end{aligned}
\right.
\ee

\subsubsection{Regularized distance and its Hessian}
\label{sub-rdf.2}
Consider the following function $\fb: \mS^{n-1}\times I\goto\R,$
\[\fb:=\frac{\phi(r-\rho(z))}{\phi(\rho(z))}.\]
Then we can see that on $\Gamma=\p\Omega=(z, \rho(z))$ we have $\fb=0.$ It is straightforward to verify that $\fb$ is a regularized distance for $\Omega.$

In this subsection, we will compute the Hessian of $\fb$ at the point $(z_0, r)\in\N^n(K)$ under the normal coordinates at $z_0$ chosen in Subsection
\ref{sub-rdf.1}. A straightforward calculation yields
\[\fb_1=-\lt(\frac{\dot\phi(r-\rho)}{\pro}+\frac{\phi(r-\rho)\dpro}{\prot}\rt)\rho_1=-\frac{\pr}{\prot}\rho_1,
\,\,\fb_\al=0\,\,\mbox{for}\,\,2\leq\al\leq n-1,\]
\[\fb_{11}=2\frac{\pr\dpro}{\phi^3(\rho)}\rho^2_1-\frac{\pr}{\prot}\rho_{11},
\,\,\fb_{1\al}=-\frac{\pr}{\prot}\rho_{1\al}\,\,\mbox{for}\,\,2\leq\al\leq n-1,\]
\[\fb_{\al\beta}=-\frac{\pr}{\prot}\rho_{\al\beta}\,\,\mbox{for}\,\,2\leq\al, \beta\leq n-1,\]
and
\[\fb_r=\frac{\dot\phi(r-\rho)}{\pro}, \fb_{rr}=-K\fb. \]

Now, let $\tau_{a}=\frac{e_{a}}{\pr}, 1\leq a\leq n-1,$ then $\{\tau_1, \cdots, \tau_{n-1}, \tau_r\}$ forms an orthonormal frame at $(z_0,r)$. Combining \eqref{hess1.1}, \eqref{hess1.2}, and \eqref{hess1.3} with \eqref{second-fundamental-form} we get, at the point $(z_0, r)$
\begin{eqnarray}
\bn^2_{11}\fb&=&\bn^2\fb(\tau_1,\tau_1)=\frac{1}{\prt}\fb_{11}+\frac{\dpr}{\pr}\fb_r=\frac{w^3}{\pr}a_{11}-K\fb, \label{cs2.1}\\
\bn^2_{1\al}\fb&=&\bn^2\fb(\tau_1,\tau_{\al})=\frac{1}{\prt}\fb_{1\al}=\frac{w^2}{\pr}a_{1\al}, \label{cs2.2}\\
\bn^2_{\al\beta}\fb&=&\bn^2\fb(\tau_{\al},\tau_{\beta})=\frac{1}{\prt}\fb_{\al\beta}+\frac{\dpr}{\pr}\fb_r\delta_{\al\beta}
=\frac{w}{\pr}a_{\al\beta}-K\fb\delta_{\al\beta}, \label{cs2.3}\\
\bn^2_{rr}\fb&=&\bn^2\fb(\tau_r, \tau_r)=-Kb, \label{cs2.4}\\
\bn^2_{ij}\fb&=&\bn^2\fb(\tau_i,\tau_j)=0,\,\, \mbox{for all other cases,}\nonumber
\end{eqnarray}
where $2\leq\al, \beta\leq n-1.$
Therefore, at $(z_0, r)$ under the normal coordinates chosen in Subsection
\ref{sub-rdf.1} we obtain
\be\label{hessian-b}
\begin{aligned}
\text{Hessian}(\fb)&=\left[\ju{ccccc}{\frac{w^3}{\pr}a_{11}&\frac{w^2}{\pr}a_{12}&\cdots&\frac{w^2}{\pr}a_{1n-1}&0\\
\frac{w^2}{\pr}a_{12}&\frac{w}{\pr}a_{22}&\cdots&0&0\\
\vdots&\vdots&\ddots&\vdots\\
\frac{w^2}{\pr}a_{1n-1}&0&\cdots&\frac{w}{\pr}a_{n-1n-1}&0\\
0&0&\cdots&0&0}\right]-K\fb I_n.
\end{aligned}
\ee
where $I_n$ is the $n\times n$ identity matrix.
We want to emphasize that $\kappa[a_{ij}]=(\kappa_1, \cdots, \kappa_{n-1})$ are the principal curvatures of $\Gamma$ at $(z_0, \rho(z_0)).$

In general, consider the function $\psi=\psi(\mathfrak b),$ where $\psi$ is a function defined on $\R.$
We will compute the Hessian of $\psi$ at $(z_0, r).$ Denote $\psi'|_{(z_0, r)}=\frac{d\psi}{d\fb}|_{(z_0, r)}=:A,$ $\psi''|_{(z_0, r)}=\frac{d^2\psi}{d\fb^2}|_{(z_0, r)}=:B,$ and $\tau_n:=\tau_r.$ We obtain for
$1\leq i,j\leq n$,
$$\bn^2_{ij}\psi=A\bn^2_{ij}\fb+B(\tau_i\fb)(\tau_j\fb).$$
From \eqref{hessian-b}, it is easy to see that at $(z_0, r),$ under the specific coordinates we chose earlier, the hessian of $\psi$ is
 \be\label{hessian-psi}
\begin{aligned}
\text{Hessian}(\psi)&=-AK\fb I_n+\left[\ju{ccccc}{\frac{Aw^3}{\pr}a_{11}+B\frac{\rho_1^2}{\phi^4(\rho)}&\frac{Aw^2}{\pr}a_{12}&\cdots&\frac{Aw^2}{\pr}a_{1n-1}&
-B\frac{\dot\phi(r-\rho)\rho_1}{\phi^3(\rho)}\\
\frac{Aw^2}{\pr}a_{12}&\frac{Aw}{\pr}a_{22}&\cdots&0&0\\
\vdots&\vdots&\ddots&\vdots\\
\frac{Aw^2}{\pr}a_{1n-1}&0&\cdots&\frac{Aw}{\pr}a_{n-1n-1}&0\\
-B\frac{\dot\phi(r-\rho)\rho_1}{\phi^3(\rho)}&0&\cdots&0&B\frac{\dot\phi^2(r-\rho)}{\prot}}\right]\\
&=: -AK\fb I_n+M
\end{aligned}
\ee

\section {Application}
\label{app}
We will describe some applications of the regularized distance $\fb.$ More precisely, we will use $\fb$ to construct barrier functions and establish
$C^{1, 1}$ estimates for the admissible solution of the following equation:
\be\label{app-main}
\left\{\begin{aligned}
F\lt(\bn^2u, u\rt):=\s_k^{\frac{1}{k}}\lt(\lambda\lt(\bn^2u+Kug^\N\rt)\rt)&=0\,\,&\mbox{in}\,\,\mathcal{M}\\
u&=\frac{K+1}{2}\,\,&\mbox{on}\,\,\p\Omega_0\\
u&=\frac{K-1}{2}\,\,&\mbox{on}\,\,\p\Omega_1,
\end{aligned}\right.
\ee
where $\mathcal M=\Omega_0\setminus\bar\Omega_1,$ $\p\Omega_i=\{(z, \rho^i(z)): z\in\mS^{n-1}\}\subset\N^n(K)$ for $i=0, 1,$ and $\rho^0(z)>\rho^1(z)$ for all $z\in\mS^{n-1}.$

\begin{proposition}
\label{app-prop1}
Let $\mathcal M$ be the domain given above. We have the following conclusions.\\

 (1). When $\N^n=\R^n,$ $\p\Omega_i$ are strictly $(k-1)$-convex, then there is a subsolution $\lu$ of \eqref{app-main} such that $-\frac{1}{2}\leq\lu\leq \frac{1}{2},$
 $\lu=\frac{1}{2}$ on $\p\Omega_0,$ $\lu=-\frac{1}{2}$ on $\p\Omega_1,$ and $\s_k^{\frac{1}{k}}(\bn^2 \lu)>0$ in $\mathcal M.$ Moreover, $\lu$ satisfies
 $\bn_\nu\lu>0$ on $\p\mathcal M,$ where $\nu$ is the unit normal to $\p\mathcal M$ that points away from the origin.\\

 (2). When $\N^n=\mH^n,$ $\p\Omega_i$ are strictly $k$-convex and $\max\limits_{z\in\mS^{n-1}}(\rho^0(z)-\rho^1(z))<\tanh^{-1}\lt(\frac{1}{\sqrt{n-2}}\rt)$, then there is a subsolution $\lu$ of \eqref{app-main} such that $-1\leq\lu\leq 0,$ $\lu=0$ on $\p\Omega_0,$ $\lu=-1$ on $\p\Omega_1,$ and $\s_k^{\frac{1}{k}}(\bn^2 \lu-\lu g^\N)>0$ in $\mathcal M.$ Moreover, $\lu$ satisfies
 $\bn_\nu\lu>0$ on $\p\mathcal M,$ where $\nu$ is the unit normal to $\p\mathcal M$ that points away from the origin.\\

 (3). When $\N^n=\mS^n,$ $\p\Omega_i$ are strictly $k$-convex, $\mathcal M\ssubset\mS^n_+$ and $\max\limits_{z\in\mS^{n-1}}\frac{\sin\lt(\rho^0(z)-\rho^1(z)\rt)}{\sin\lt(\rho^1(z)\rt)}<\gamma$ for some $\gamma=\gamma(\p\Omega_1, \p\Omega_0)>0$ sufficiently small, then there is a subsolution $\lu$ of \eqref{app-main} such that $0\leq\lu\leq 1,$ $\lu=1$ on $\p\Omega_0,$ $\lu=0$ on $\p\Omega_1,$ and $\s_k^{\frac{1}{k}}(\bn^2 \lu+\lu g^\N)>0$ in $\mathcal M.$ Moreover, $\lu$ satisfies
 $\bn_\nu\lu>0$ on $\p\mathcal M,$ where $\nu$ is the unit normal to $\p\mathcal M$ that points away from the origin.\\
\end{proposition}

\begin{proposition}
\label{app-prop2}
Let $\mathcal M$ be the domain given above. We have the following conclusions.\\

 (1). When $\N^n=\R^n\,\, \mbox{or}\,\,\mH^n,$ there is an upper barrier $\bar u$ of the admissible solution $u$ of \eqref{app-main} which satisfies
 \[\bar{\Laplace}\bar u+nK\bar u<0\,\,\mbox{in}\,\,\mathcal M,\]
 $\bar{u}=\frac{K+1}{2}$ on $\p\Omega_0,$ and $\bar{u}=\frac{K-1}{2}$ on $\p\Omega_1.$ Moreover, $\bar u$ satisfies
 $\bn_\nu\bar u>0$ on $\p\mathcal M,$ where $\nu$ is the unit normal to $\p\mathcal M$ that points away from the origin. \\

(2). When $\N^n=\mS^n,$ suppose $\mathcal M\ssubset\mS^n_+$ and
 $\max\limits_{z\in\mS^{n-1}}\frac{\sin\lt(\rho^0(z)-\rho^1(z)\rt)}{\sin\lt(\rho^1(z)\rt)}<\gamma$ for some $\gamma=\gamma(\p\Omega_1, \p\Omega_0)>0$ sufficiently small. Then there is an upper barrier $\bar u$ of the admissible solution $u$ of $\bar u$ of \eqref{app-main} which satisfies
 \[\bar{\Laplace}\bar u+n\bar u<0\,\,\mbox{in}\,\,\mathcal M,\]
 $\bar{u}=1$ on $\p\Omega_0,$ and $\bar{u}=0$ on $\p\Omega_1.$ Moreover, $\bar u$ satisfies
 $\bn_\nu\bar u>0$ on $\p\mathcal M,$ where $\nu$ is the unit normal to $\p\mathcal M$ that points away from the origin.
\end{proposition}

\begin{remark}
We remark that, when $\N^n=\mS^n,$ in order to apply the maximum principle to control the height of $u,$ one needs the condition $\la_1(\mathcal M)>n,$ where
$\la_1(\mathcal M)$ is the first Dirichlet eigenvalues of the Laplace operator on $\mathcal M.$ Since it is well known that $\la_1(\mS^n_+)=n,$ by assuming
$\mathcal M\ssubset\mS^n_+,$ the desired condition can be easily verified using Theorem 1.3 in \cite{Har24}.
\end{remark}

The existence of a smooth subsolution $\lu$ of \eqref{app-main} and a smooth upper barrier $\bar u$ of the admissible solution of \eqref{app-main} is crucial in the proof of the main theorem. We will prove Proposition
\ref{app-prop1} and \ref{app-prop2} in the next section. Following the idea of \cite{Guan02}, to establish the existence of the solution $u\in C^{1, 1}(\mathcal M)$ of \eqref{app-main}, we consider the following equation with parameter $0\leq t<1,$
\be\label{app-appr}
\left\{\begin{aligned}
F\lt(\bn^2u, u\rt):=\s_k^{\frac{1}{k}}\lt(\lambda\lt(\bn^2u+Kug^\N\rt)\rt)&=(1-t)f^K_0=:f^K_t\,\,&\mbox{in}\,\,\mathcal{M}\\
u&=\frac{K+1}{2}\,\,&\mbox{on}\,\,\p\Omega_0\\
u&=\frac{K-1}{2}\,\,&\mbox{on}\,\,\p\Omega_1,
\end{aligned}\right.
\ee
where  $f_0^K=\s_k^{\frac{1}{k}}\lt(\bn^2\lu+K\lu g^\mathbb N\rt),$
and $\lu$ is as in Proposition \ref{app-prop1}. We will prove that when $\N^n=\R^n$ and $\mH^n,$ equation \eqref{app-appr} admits a smooth admissible solution with a uniform $C^2$ bound. When $\N^n=\mS^n,$ any admissible solution $u^t$ of \eqref{app-appr} satisfying $u^t\geq\lu $ also enjoys a uniform $C^2$ bound.
\begin{theorem}
\label{app-thm1}
Let $\mathcal M$ be a domain that satisfies the conditions of Proposition \ref{app-prop1} and \ref{app-prop2}. For each $0\leq t<1,$ let $u^t$ be an admissible solution of \eqref{app-appr} satisfying $u^t\geq\lu.$ Then we have
\be\label{app-c2}
\|u^t\|_{C^{2}(\mathcal M)}\leq C,
\ee
where $C=C(\mathcal M)>0$ is a positive constant that depends only on $\mathcal M$ (independent of $t$).
\end{theorem}

Note that when $K=0, -1$ we obtain $u^t\geq \lu$ from the standard maximum principle. For our convenience, in the rest of the proof, we will drop the superscript $t.$

\begin{proof}[\textit{Proof of Theorem \ref{app-thm1}}]
{\it $C^0$-estimates.} The lower bound for $u$ comes from the assumption that $u\geq \lu$ while the upper bound for $u$ follows from the maximum principle. More precisely,
$u\leq\bar u$ for all $x\in\mathcal M.$

{\it $C^1$-estimates.} We will employ the same test function as the one used in Section 5 of \cite{Guan99}. Set
\[\omega:=\lt(1+|K|\sup\limits_{\mathcal M}u^2+Ku^2+|\bn u|^2\rt)^{1/2}\]
and
\[v=1+u,\]
 we will consider $$W=\max\limits_{\bar{\mathcal M}}\omega\exp(av^2).$$ Here, $a>0$ is a small constant to be determined later. We assume $W$ is attained at an interior point
 $x_0\in\mathcal M.$ Choose a local orthonormal frame $e_1, \cdots, e_n$ around $x_0$ and differentiate the function $\log\omega+av^2$ at $x_0,$ we have
 \be\label{app-1}
 \frac{\omega_i}{\omega}+2avu_i=0.
 \ee
 We may rotate the coordinates such that at $x_0$
 \[u_1(x_0)=|\bn u(x_0)|\,\,\mbox{and}\,\,u_{\al\beta}(x_0)=u_{\al\al}(x_0)\delta_{\al\beta},\,\,\mbox{for}\,\, 2\leq\al, \beta\leq n.\]
 Denote $h_{ik}:=u_{ik}+Ku\delta_{ik},$ in view of the commutation formula
 $$\bn_i\bn_i\bn_ju=:u_{jii}=u_{iij}+Ku_j\,\,\mbox{for}\,\,i\neq j,$$ we can derive
 $h_{kii}=h_{iik}.$ Recall \eqref{app-1} we see that
 \[2\omega\omega_i=2Kuu_i+2\sum_lu_lu_{li}=2\sum_lu_lh_{li}=-4av\omega^2u_i.\]
 Therefore, for $i\geq 2$ we have
 \[2\sum_lu_lh_{li}=2u_1h_{1i}=0.\] This yields
 $u_{ij}(x_0)$ is diagonalized at $x_0$ and $2u_1h_{11}=-4av\omega^2u_1<0.$
 {We note that, when $k=n$ in \eqref{app-appr}, that is, the PDE we consider is $\s_n^{\frac{1}{n}}\lt(\bn^2u+Kug^N\rt)=(1-t)f^K_0,$
 we obtain a contradiction directly. Therefore, in what follows we always assume $2\leq k\leq n-1.$}

We will denote $F^{ij}:=\frac{\p F}{\p h_{ij}}.$ Applying the second
derivative test at $x_0$ we obtain
\be\label{app-2}
0\geq\sum_iF^{ii}\frac{\omega_{ii}}{\omega}+(2a-4a^2v^2)\sum_iF^{ii}u^2_i+2av\lt(f_t^K-Ku\sum_iF^{ii}\rt).
\ee
A straightforward calculation gives
\[\omega^2_i+\omega\omega_{ii}=\sum_lu_{li}h_{li}+\sum_lu_lh_{lii}
=\sum_l(h_{li}-Ku\delta_{li})h_{li}+\sum_lu_{l}h_{lii}.\]
Therefore,  we have
\[\omega\omega_{ii}=\sum_l(h_{li}-Ku\delta_{li})h_{li}+\sum_lu_{l}h_{lii}-\frac{\lt(\sum_lu_lh_{li}\rt)^2}{\omega^2}
\geq-Kuh_{ii}+u_lh_{lii}.\]
Plugging this into \eqref{app-2} we obtain
\be\label{app-3}
0\geq\frac{-Kuf_{t}^{K}+u_1(f_{t}^{K})_1}{\omega^2}+(2a-4a^2v^2)F^{11}u^2_1+2avf^K_t-2avKu\sum_iF^{ii}.
\ee
By equation (3.10) of \cite{CW01} we know there exists $\theta=\theta(n, k)>0$ such that
$F^{11}\geq\theta\sum\limits_iF^{ii}.$ Choosing $a>0$ so small that $2a-4a^2v^2\geq a$ then \eqref{app-3} implies
\[|\bn u(x_0)|\leq C,\]
for some $C=C(|u|_{C^0}, |f^K_0|_{C^1}, n, k)>0.$ This yields if $W$ is attained in the interior of $\mathcal M,$ then
$W<C.$

On the other hand, if $W$ is attained at $\p\mathcal M,$ since
\[\lu\leq u\leq\bar u\] and
\[u\mid_{\p\mathcal M}=\lu\mid_{\p\mathcal M}=\bar u\mid_{\p\mathcal M},\] we immediately conclude
\be\label{c1-bdry} |\bn u(x)|\leq\max\{|\bn\lu(x)|, |\bn\bar u(x)|\}<C,\ee
for all $x\in\p\mathcal M.$
Therefore, we have $W\leq C$ which yields $\max\limits_{\bar{\mathcal M}}|\bn u(x)|\leq C.$

{\it $C^2$-estimates.} In order to obtain the $C^2$ a priori estimates for $u,$ we need to get the estimates of the second derivatives of
$u$ on $\p\mathcal M$ first.

Let $\xi, \eta$ be some tangential unit vector fields on $\p\mathcal M,$ since $u\equiv\text{constant}$ on each connected component of $\p\mathcal M$
we have
\be\label{add1}\bn_{\xi\eta}u=-\bn_{\nu}uII(\xi, \eta)\,\,\mbox{on}\,\,\p\mathcal M=\p\Omega_0\cup\p\Omega_1,\ee
where $\nu$ is the unit normal on $\p\mathcal M$ that points away from the origin and $II$ denotes the second fundamental form of $\p\mathcal M.$

The second derivative on the outside boundary of $\p\mathcal M,$ namely $\p\Omega_0,$ has more or less been done in \cite{Guan99} (see also \cite{Guan14}).
Therefore, we will concentrate on obtaining the $C^2$ estimates on $\p\Omega_1.$

Let $\e>0$ be a small constant, denote
\[\Omega_1^\e:=\{x\in\mathcal M: \text{dist}(x, \p\Omega_1)<\e\}.\]
 We also denote
\[\psi^\e_K:=\max\limits_{\Omega^\e_1}F(\bn_{ij}\lu+K\lu g^\N).\]
\textbf{Claim:} There exists a local subsolution $\lul$ that satisfies
\[F(\bn_{ij}\lul+K\lul g^\N)>2\psi^\e_K\,\,\mbox{in}\,\,\Omega^{\td\e}_1,
\,\,\lul=\frac{K-1}{2}\,\,\mbox{on}\,\,\p\Omega_1,\]
and
\[\lul<\lu\,\,\mbox{in}\,\, \Omega^{\td\e}_1, \]
where $0<\td\e\leq\e$ and
$\Omega_1^{\td\e}:=\{x\in\mathcal M: \text{dist}(x, \p\Omega_1)<\td\e\}.$
This \textbf{Claim} will be proved in Subsection \ref{subsec-loc}.

We emphasize that the reason for introducing a local subsolution is to derive a $C^2$ estimate that is independent of $t$; this is necessary since the solution to \eqref{app-appr} starts from $\lu$ (when t = 0).

In view of the \textbf{Claim}, there exists some $\theta>0$
small such that
\be\label{app-4}F(\bn_{ij}(\lul-\theta\Phi)+K(\lul-\theta\Phi)g^\N)>\frac{3}{2}\psi^\e_K\,\,\mbox{in}\,\, \Omega^{\td\e}_1,\ee
where $\Phi=\int_0^r\phi(s)ds.$ We note that by \cite{Pet} page 8 we have
\[\bn_{ij}\Phi=\dot\phi g^\N_{ij}.\] We consider the linear operator
$\mathfrak{L}$ defined by
\[\mathfrak Lv=F^{ij}\bn_{ij}v+K^-v\sum\limits_iF^{ii}\,\,\mbox{for}\,\, v\in C^2(\Omega^{\td\e}_1),\]
where $K^-=\min\{K, 0\}.$
In view of \eqref{app-4} we get
\[\mathfrak L(u-(\lul-\theta\Phi))<-\frac{1}{2}\psi^\e_K,\]
which gives
\be\label{app-5}
\mathfrak L(u-\lul)<-C_1\sum\limits_iF^{ii}.
\ee
Here we have used $\sum F^{ii}\geq \s_k^{1/k}(\vec{1})=\lt(C_n^k\rt)^{1/k}.$
We will denote $v:=u-\lul,$ then $v$ satisfies
\[\mathfrak Lv\leq-C_1\sum_iF^{ii}\,\,\mbox{in}\,\,\Omega^{\td\e}_1
\,\,\mbox{and}\,\,v\geq 0\,\,\mbox{on}\,\,\bar\Omega^{\td\e}_1.\]
Differentiating \eqref{app-appr} and applying the commutation formula we get
\[F^{ij}\lt(u_{ijk}+Ku_k\delta_{ij}\rt)=F^{ij}\lt(u_{kij}+Ku_j\delta_{ik}\rt)=(f^K_t)_k.\]
This yields
\be\label{app-6}
\lt|\mathfrak L\bn_ku\rt|\leq C\sum_iF^{ii},
\ee
Now, choose any point $x_0\in\p\Omega_1,$ let $d_0(x)$ be the distance from $x$ to $x_0.$ Set
\[\td B_\delta:=\{x\in\mathcal M: d_0(x)<\delta\}.\]
Since $\bn_{ij}d^2_0(x_0)=2\delta_{ij},$ by choosing $0<\delta<\td\e$ sufficiently small,
we may assume $d_0$ is a smooth function in $\td B_\delta$ and
\[\{\delta_{ij}\}\leq\{\bn_{ij}d^2_0\}\leq 3\{\delta_{ij}\}\,\,\mbox{in}\,\,\td B_\delta.\]
In the following, we will adapt the idea of Ivochkina \cite{Ivo90} (see also \cite{Guan14}) to prove the $C^2$ boundary estimates in the tangential-normal direction.

Let $\xi$ be an arbitrary $C^2$ vector field defined in $\td B_\delta$ and $\eta$ be the parallel transport of the unit inward normal vector field $\nu$ to $\p\Omega_1\cap\p\td B_\delta$ along geodesics that intersect $\p\Omega_1$ orthogonally in
$\td B_\delta.$ We define
\[W:=\lt<\bn u, \xi\rt>\,\,\mbox{and}\,\,V: =\frac{1}{2}\lt|\bn u\rt|^2-\frac{1}{2}\lt<\bn u, \eta\rt>^2.\]
Since $u\equiv \text{Constant}$ on $\p\Omega_1,$ by our choice of $\eta,$ we can see that $V=0$ on $\p\Omega_1\cap\p\td B_\delta.$
We compute the derivatives of $W$ and obtain
\be\label{app-7}
\bn_iW=\lt<\bn_i\bn u, \xi\rt>+\lt<\bn u, \bn_i\xi\rt>
\ee
and
\be\label{app-8}
\bn_j\bn_iW=\xi^k\bn_{jik}u+\bn_j\xi^k\bn_{ik}u+\bn_i\xi^k\bn_{jk}u
+\lt<\bn u, \bn_{ji}\xi\rt>.
\ee
where $\xi^k$ is the $k$-th component of $\xi$.
Now let $B=\{b_{ij}\}$ be an orthogonal matrix that simultaneously diagonalizes $\{F^{ij}\}$ and $\{h_{ij}\}.$ We denote the eigenvalues of $\{F^{ij}\}$ and $\{h_{ij}\}$
by $(f_1, \cdots, f_n)$ and $(\la_1, \cdots, \la_n)$ respectively. Then we have
\[h_{ij}=\sum\limits_l\la_lb_{li}b_{lj}\,\,\mbox{and}\,\, F^{ij}=\sum\limits_lf_lb_{li}b_{lj}.\]
This implies for an arbitrary matrix $A=\{a_{jk}\}$ we get
\[
\begin{aligned}
\sum\limits_{i,j,k}F^{ij}a_{jk}h_{ki}&=\sum_{i, j, k}\lt(\sum_sf_sb_{si}b_{sj}\rt)a_{jk}\lt(\sum_l\la_lb_{lk}b_{li}\rt)\\
&=\sum_{l, j, k}f_l\la_lb_{lj}a_{jk}b_{lk}.
\end{aligned}
\]
Therefore, by virtue of \eqref{app-6} we have
\be\label{app-9}
\lt|\mathfrak L W\rt|<C\lt(\sum f_i|\la_i|+\sum f_i\rt),
\ee
where $\la_i, 1\leq i\leq n,$ are the eigenvalues of the matrix $(\bn_{ij}u+Ku\delta_{ij}).$
Next, we compute the derivatives of $V$ and obtain
\be\label{app-10}
\bn_iV=\bn_k u\bn_{ik}u-\lt<\bn u, \eta\rt>\lt(\eta^k\bn_{ik}u+\bn_ku\bn_i\eta^k\rt)
\ee
and
\be\label{app-11}
\begin{aligned}
\bn_j\bn_iV&=\bn_ku\bn_{jik}u+\lt(\bn_{jk}u-\lt<\bn_j\bn u, \eta\rt>\eta^k\rt)\bn_{ik}u\\
&-\lt<\bn_j\bn u, \eta\rt>\bn_ku\bn_i\eta^k-\lt<\bn u, \bn_j\eta\rt>\lt(\eta^k\bn_{ik}u+\bn_ku\bn_i\eta^k\rt)\\
&-\lt<\bn u, \eta\rt>\lt(\bn_j\eta^k\bn_{ik}u+\eta^k\bn_{jik}u+\bn_{jk}u\bn_i\eta^k+\bn_ku\bn_{ji}\eta^k\rt)
\end{aligned}
\ee
Now, for any point $x\in\td B_\delta,$ we may let $\{e_1, \cdots, e_n\}$ be an orthonormal frame at $x$ with $\eta(x)=e_n.$
By Proposition 2.19 of \cite{Guan14} we have
\[
\begin{aligned}
&F^{ij}\lt(\bn_{jk}u-\lt<\bn_j\bn u, \eta\rt>\eta^k\rt)\bn_{ik}u\\
&=\sum_{l<n}F^{ij}u_{il}u_{jl}\geq c_0\sum_{i\neq r}f_i\lt(\la_i-Ku\rt)^2\\
&>c_0\sum_{i\neq r}f_i\la_i^2-C\lt(\sum f_i|\la_i|+\sum f_i\rt).
\end{aligned}
\]
We conclude that
\be\label{app-12}
\mathcal LV>c_0\sum_{i\neq r}f_i\la_i^2-C\lt(\sum f_i|\la_i|+\sum f_i\rt).
\ee
Below, we assume that the vector field $\xi$ is the tangent vector along $\p\Omega_1\cap\p\td B_\delta.$ Then $W=0$ on $\p\Omega_1\cap\p\td B_{\delta}.$ Following the argument in \cite{Guan14}, let
$$\Psi=A_1v+A_2d^2_0-A_3V.$$
In view of Corollary 2.21 of \cite{Guan14} we have, when $A_1\gg A_2\gg A_3\gg 1$
\[\mathfrak L\lt(\Psi\pm W\rt)\leq 0\,\,\mbox{in}\,\,\td B_\delta\]
and $\Psi\pm W\geq 0$ on $\p\td B_\delta.$ By the maximum principle we derive $\Psi\pm W\geq 0$ in $\td B_\delta$ and therefore
\be\label{app-13}
|\bn_{n\al}u(x_0)|\leq\bn_n\Psi(x_0)\leq C,\,\,\mbox{for}\,\,\forall \al<n.
\ee
To estimate $u_{nn}(x_0),$ we note that on $\p\Omega_1$
\[\s_{k-1}(u_{\al\beta}+Ku\delta_{\al\beta})(u_{nn}+Ku)+\s_k(u_{\al\beta}+Ku\delta_{\al\beta})-\sum_{s=1}^{n-1}u^2_{ns}\s_{k-2}(u_{\al\beta}+Ku\delta_{\al\beta})=\lt(f^K_t\rt)^k,\]
where $1\leq\alpha, \beta\leq n-1.$
Since $u\geq\lu$ in $\mathcal M$ and $u=\lu$ on $\p\mathcal M,$ in view of Proposition \ref{app-prop1} we have $u_\nu\geq\lu_\nu\geq c>0$ on $\p\Omega_1,$ where $\nu$ is the unit normal to $\p\Omega_1$ that points away from the origin. We note that $\lu_\nu>c$ comes from the construction of $\lu,$ which will be given in Subsection \ref{subsec-lu}.
In view of \eqref{add1} and our assumptions on $\p\Omega_1,$ it is easy to see that $\s_{k-1}(u_{\al\beta}+Ku\delta_{\al\beta})>c_0>0,$ this implies
\be\label{c2-nn}|u_{nn}(x_0)|<C.\ee
We want to emphasize that the property $\bar u_\nu>0$ in $\p\mathcal M,$ as stated in Proposition \ref{app-prop2}, is essential for obtaining a uniform bound of $|u_{nn}|$ on $\p\Omega_0.$

Next, we will establish $C^2$ global estimates. Following \cite{Guan99}, set
 \[W=\lt(\bn^2_{\xi\xi} u+Ku\rt)\exp\lt\{\frac{a}{2}\lt(|\bn u|^2+Ku^2\rt)+b\Phi\rt\},\]
 where $|\xi|=1,$ $\xi\in T_x\N^n,$ $\Phi=\int_0^r\phi(s)ds,$ and $a, b>0$ are some constant to be determined. We may assume $W$ achieves its maximum at an interior point $x_0\in\mathcal M$ for some unit vector $\xi\in T_x\N^n.$
 Choose a smooth orthonormal local frame $\{e_1, \cdots, e_n\}$ at $x_0$ such that $e_1(x_0)=\xi$ and $\{\bn^2_{ij}u(x_0)\}$ is diagonal.
 We denote $\la_i:=\bn^2_{ii}u(x_0)+Ku(x_0)=h_{ii}.$ We will always assume $\la_1\geq\cdots\geq \la_n.$ Differentiating $W$ at $x_0$
we obtain
\be\label{app-crit}
0=\frac{h_{11i}}{\la_1}+au_i\la_i+b\Phi_i
\ee
and
\be\label{app-sec}
\begin{aligned}
0&\geq\frac{1}{\la_1}F^{ii}h_{11ii}-F^{ii}\lt(\frac{h_{11i}}{\la_1}\rt)^2+aF^{ii}u_{ii}\la_i\\
&+aF^{ii}u_lu_{lii}+KaF^{ii}u^2_i+bF^{ii}\Phi_{ii}\\
&=\frac{1}{\la_1}F^{ii}h_{11ii}-F^{ii}\lt(\frac{h_{11i}}{\la_1}\rt)^2+aF^{ii}\la^2_i-aKuf^K_t\\
&+aF^{ii}u_l\lt(u_{iil}+Ku_l-Ku_i\delta_{li}\rt)+aKF^{ii}u^2_i+bF^{ii}\Phi_{ii}\\
&\geq\frac{1}{\la_1}F^{ii}h_{11ii}-F^{ii}\lt(\frac{h_{11i}}{\la_1}\rt)^2+aF^{ii}\la^2_i-aKuf^K_t
+au_l(f^K_t)_l+b\dot\phi\sum F^{ii}.
\end{aligned}
\ee
Note that
\[
\begin{aligned}
F^{ii}h_{11ii}&=F^{ii}(u_{11ii}+Ku_{ii})\\
&=F^{ii}[u_{ii11}+2K(u_{11}-u_{ii})+Ku_{ii}]\\
&=F^{ii}(u_{ii11}+Ku_{11})+KF^{ii}(u_{11}-u_{ii})\\
&=(f^K_t)_{11}-Kf^K_t-F^{ij, kl}h_{ij1}h_{kl1}+K\la_1\sum F^{ii}.
\end{aligned}
\]
Plugging this into \eqref{app-sec} we have
\be\label{app-sec1}
\begin{aligned}
0&\geq-\frac{1}{\la_1}F^{ij, kl}h_{ij1}h_{kl1}+K\sum F^{ii}\\
&-F^{ii}\lt(\frac{h_{11i}}{\la_1}\rt)^2+aF^{ii}\la^2_i+b\dot\phi\sum F^{ii}-C,
\end{aligned}
\ee
where $C=C(|u|_{C^1}, |f^K_0|_{C^2})>0.$ Below, we divide the analysis of \eqref{app-sec1} into two cases.

\textbf{Case A.} At $x_0,$ $\la_n\leq-\frac{\la_1}{n}.$ By virtue of \eqref{app-crit} we get
\[\lt(\frac{h^2_{11i}}{\la_1}\rt)^2\leq 2a^2u^2_i\la^2_i+2b^2\Phi^2_i.\] Combining this with the concavity of $F,$
\eqref{app-sec1} becomes
\be\label{app-secA1}
\begin{aligned}
0&\geq K\sum F^{ii}+\lt[a-2a^2|\bn u|^2\rt]\sum F^{ii}\la^2_i\\
&+\lt[b\dot\phi-2b^2|\bn\Phi|^2\rt]\sum F^{ii}-C.
\end{aligned}
\ee
We will choose $a>0$ so small that
\[a-2a^2|\bn u|^2\geq \frac{a}{2}.\]
Since in this case we have
\[\sum F^{ii}\la^2_i\geq F^{nn}\la^2_n\geq\frac{\la^2_1}{n^3}\sum F^{ii},\]
\eqref{app-secA1} gives
\[0\geq\frac{a}{2}\frac{\la_1^2}{n^3}+\lt[b\dot\phi-2b^2|\bn\Phi|^2+K\rt]-C.\]
Therefore, in this case we have
\[\la_1(x_0)<C=C(|u|_{C^1}, |f^K_0|_{C^2}, \mathcal M, n).\]
We note that here $\la_1$ is also determined by $b.$ However, as we will see in the discussion of \textbf{Case B}, $b=b(|u|_{C^1}, |f^K_0|_{C^2}, \mathcal M)>0$ is chosen to be a bounded constant. Thus, $\la_1(x_0)$ is bounded in this case.

\textbf{Case B.} At $x_0,$ $\la_n>-\frac{\la_1}{n}.$ Let us partition $\{1, \cdots, n\}$ into two parts.
\[I=\{j: f_j\leq n^2 f_1\}\,\,\mbox{and}\,\, J=\{j: f_j> n^2f_1\}.\] Then we have
\[
\begin{aligned}
\sum_{i\in I}F^{ii}\lt(\frac{h_{11i}}{\la_1}\rt)^2&\leq 2b^2|\bn\Phi|^2\sum_{i\in I}f_i+2a^2|\bn u|^2\sum_{i\in I}f_i\la^2_i\\
&\leq 2n^3b^2|\bn\Phi|^2f_1+2a^2|\bn u|^2\sum_if_i\la^2_i.
\end{aligned}
\]
Moreover, by the concavity of $F$ we obtain (see (5.38) of \cite{GSS} for example)
\[
\begin{aligned}
-\frac{1}{\la_1}F^{ij, kl}h_{ij1}h_{kl1}&\geq\frac{2}{\la_1}\sum_{i\in J}\frac{f_i-f_1}{\la_1-\la_i}h^2_{11i}\\
&\geq\frac{2}{\la^2_1}\lt(1-\frac{1}{n}\rt)\sum_{i\in J}f_ih^2_{11i}\geq\sum_{i\in J}f_i\lt(\frac{h^2_{11i}}{\la_1}\rt)^2.
\end{aligned}
\]
Therefore, in this case \eqref{app-sec1} gives
\be\label{app-secB1}
\begin{aligned}
0&\geq\lt[a-2a^2|\bn u|^2\rt]f_1\la_1^2-2n^3b^2|\bn\Phi|^2f_1+\lt[b\dot\phi+K-C\rt]\sum F^{ii}\\
&\geq\lt[\frac{a}{2}\la^2_1-2n^3b^2|\nb\Phi|^2\rt]f_1+\lt(b\dot\phi+K-C\rt)\sum F^{ii}.
\end{aligned}
\ee
By choosing $b>C=C(|u|_{C^1}, |f^K_0|_{C^2}, \mathcal M)>0$ large we obtain
\[\la_1(x_0)<C=C(|u|_{C^1}, |f^K_0|_{C^2}, \mathcal M, n).\]
Therefore, $W$ is uniformly bounded.
\end{proof}

The existence of an admissible solution of \eqref{app-appr} follows from standard PDE theory when $\N^n=\R^n$ or $\mH^n,$ where the maximum principle applies.
Theorem \ref{thm2} then follows from the standard convergence argument.

\section {Construction of barriers}
\label{sec-con}
In this section, we will construct barriers and prove Proposition \ref{app-prop1} and \ref{app-prop2}.
Let $M:=\{m_{ij}\}$ be the matrix defined in \eqref{hessian-psi}, we will start with computing $\s_j(M).$
A direct calculation yields
\be\label{s1m}
\begin{aligned}
\s_1(M)&=\frac{Aw^3}{\pr}a_{11}+\frac{Aw}{\pr}\sum\limits_{\al=2}^{n-1}a_{\al\al}+B\frac{\dot\phi^2(r-\rho)}{\prot}+B\frac{\rho_1^2}{\phi^4(\rho)}\\
&=\frac{Aw^3}{\pr}a_{11}+\frac{Aw}{\pr}\sum\limits_{\al=2}^{n-1}a_{\al\al}+B\lt(\frac{w^2}{\prot}-K\fb^2\rt).
\end{aligned}
\ee
 Now, for $1<j\leq n,$ by equation (2.5) of \cite{Tru95} we have,
 \be\label{sjm}
 \begin{aligned}
 \s_j(M)&=m_{11}\s_{j-1}(M|1)+\s_j(M|1)-\sum_{s=2}^nm^2_{1s}\s_{j-2}(M|1s)\\
 &=m_{11}\lt[m_{nn}\s_{j-2}(M|1n)+\s_{j-1}(M|1n)\rt]+\s_j(M|1n)+m_{nn}\s_{j-1}(M|1n)\\
 &-\sum_{s=2}^{n-1}m^2_{1s}\lt[\s_{j-2}(M|1ns)+m_{nn}\s_{j-3}(M|1ns)\rt]-m^2_{1n}\s_{j-2}(M|1n)\\
 &=m_{nn}\lt[m_{11}\s_{j-2}(m_{\al\beta})+\s_{j-1}(m_{\al\beta})-\sum_{s=2}^{n-1}m^2_{1s}\s_{j-3}(m_{\al\beta}|s)\rt]\\
 &+\lt[m_{11}\s_{j-1}(m_{\al\beta})+\s_j(m_{\al\beta})-\sum_{s=2}^{n-1}m^2_{1s}\s_{j-2}(m_{\al\beta}|s)\rt]
 -\frac{B^2\rho^2_1\dot\phi^2(r-\rho)}{\phi^6(\rho)}\s_{j-2}(m_{\al\beta}),
 \end{aligned}
 \ee
 where $(M|i)$ is the matrix obtained by deleting the $i$-th row and $i$-th column of the matrix $M$ and $(m_{\al\beta})=(M|1n).$

More specifically, for $1\leq j\leq n$ we have the follows.

In $\R^n,$
\be\label{sjmr}
\begin{aligned}
\s_j(M)&=w^2\lt(\frac{Aw}{r}\rt)^{j-1}\frac{B}{\rho^2}\s_{j-1}(a_{pq})\\
&+w^2\lt(\frac{Aw}{r}\rt)^{j}\s_j(a_{pq})+\lt(\frac{Aw}{r}\rt)^j\s_j(a_{\al\beta})(1-w^2).
\end{aligned}
\ee

In $\mH^n,$
\be\label{sjmh}
\begin{aligned}
\s_j(M)&=w^2\lt(\frac{Aw}{\sinh r}\rt)^{j-1}\frac{B\cosh^2(r-\rho)}{\sinh^2(\rho)}\s_{j-1}(a_{pq})\\
&-B\fb^2\lt(\frac{Aw}{\sinh r}\rt)^{j-1}\s_{j-1}(a_{\al\beta})(w^2-1)\\
&+w^2\lt(\frac{Aw}{\sinh r}\rt)^{j}\s_j(a_{pq})-\lt(\frac{Aw}{\sinh r}\rt)^j\s_j(a_{\al\beta})(w^2-1).
\end{aligned}
\ee

In $\mS^n,$
\be\label{sjms}
\begin{aligned}
\s_j(M)&=w^2\lt(\frac{Aw}{\sin r}\rt)^{j-1}\frac{B\cos^2(r-\rho)}{\sin^2(\rho)}\s_{j-1}(a_{pq})\\
&+B\fb^2\lt(\frac{Aw}{\sin r}\rt)^{j-1}\s_{j-1}(a_{\al\beta})(w^2-1)\\
&+w^2\lt(\frac{Aw}{\sin r}\rt)^{j}\s_j(a_{pq})-\lt(\frac{Aw}{\sin r}\rt)^j\s_j(a_{\al\beta})(w^2-1).
\end{aligned}
\ee
Here and in what follows, $1\leq p, q\leq n-1$ and $2\leq\al, \beta\leq n-1.$ We adopt the convention that for any square matrix $\Lambda$ we have $\s_0(\Lambda)=1;$ moreover, if $\dim\Lambda<n\times n$ then $\s_n(\Lambda)=0.$

Before we start our construction, we also need the following lemmas.
\begin{lemma}
\label{matrix-lem}
Let $M$ be an $n\times n$ symmetric matrix and $I_n$ be the $n\times n$ identity matrix. Then, for any $c\in\R$ and $1\leq j\leq n$ we have
\[\s_j(M+cI_n)=\sum\limits_{l=0}^jC_{n-l}^{j-l}c^{j-l}\s_l(M).\]
\end{lemma}
\begin{proof}
Recall that we can express $\s_j(M+cI_n)$ as the sum of all $j\times j$ principal minors of $M+cI_n.$ In particular, the $j\times j$
principal minor has the form
\[\begin{aligned}
\det(\td M_j+cI_j)=\left|\ju{cccc}{m_{l_1l_1}+c&m_{l_1l_2}&\cdots&m_{l_1l_j}\\
m_{l_1l_2}&m_{l_2l_2}+c&\cdots&m_{l_2l_j}\\
\vdots&\vdots&\ddots&\vdots\\
m_{l_1l_j}&m_{l_2l_j}&\cdots&m_{l_jl_j}+c}\right|=\sum_{s=0}^jc^s\s_{j-s}(\td M_j).
\end{aligned}
\]
Summing over all such principal minors yields the desired result.
\end{proof}

\begin{lemma}
\label{glue-lem}(Lemma 3.1 of \cite{Guan02}) For all $\delta>0,$ there is an even function $h(t)\in C^\infty(\mathbb R),$ such that\\

(1) $h(t)\geq |t|$ for all $t\in\R,$ $h(t)=|t|$ for all $|t|\geq \delta;$\\

(2) $|h'(t)|\leq 1$ and $h''(t)\geq 0$ for all $t\in\R$ and $h'(t)\geq 0$ for all $t\geq 0.$
\end{lemma}

\subsection{Construction of subsolutions}\label{subsec-lu} In this subsection, we will construct a subsolution $\lu$ of \eqref{app-main}.
\begin{lemma}
\label{con-sub-in}
Let $\mathcal M=\Omega_0\setminus\bar\Omega_1,$ $\p\Omega_i=\{(z, \rho^i(z)): z\in\mS^{n-1}\}\subset\N^n(K)$ for $i=0, 1,$
and $\rho^0(z)>\rho^1(z)$ for all $z\in\mS^{n-1}.$ Set
$$\mathbf{\ubar{g}}:=\frac{e^{T_1\fbo}-1}{C_1}+\frac{K-1}{2},\,\,\mbox{for}\,\,\fbo:=\frac{\phi(r-\rho^1)}{\phi(\rho^1)},$$
we have the following conclusions.\\

 (1). When $\N^n=\R^n,$ $\p\Omega_1$ is strictly $(k-1)$-convex, then there exist some positive constants $C_1, T_1$ such that $\lbg<1/2$ on $\p\Omega_0,$
and $\s_j(\bn^2 \lbg)>0$ in $\mathcal M$ for all $1\leq j\leq k.$ Moreover, $\lbg$ satisfies
 $\bn_{\nu}\lbg>0$ on $\p\Omega_1,$ where $\nu$ is the unit normal to $\p\Omega_1$ that points away from the origin.\\

 (2). When $\N^n=\mH^n,$ $\p\Omega_1$ is strictly $k$-convex and $\max\limits_{z\in\mS^{n-1}}(\rho^0(z)-\rho^1(z))<\tanh^{-1}\lt(\frac{1}{\sqrt{n-2}}\rt)$, there exist some positive constants $C_1, T_1$ such that $\lbg<0$ on $\p\Omega_0,$
and $\s_j(\bn^2\lbg-\lbg g^\mH)>0$ in $\mathcal M$ for all $1\leq j\leq k.$ Moreover, $\lbg$ satisfies
 $\bn_{\nu}\lbg>0$ on $\p\Omega_1,$ where $\nu$ is the unit normal to $\p\Omega_1$ that points away from the origin.\\

 (3). When $\N^n=\mS^n,$ $\p\Omega_1$ is strictly $k$-convex, $\mathcal M\ssubset\mS^n_+,$ and
 $\max\limits_{z\in\mS^{n-1}}\frac{\sin\lt(\rho^0(z)-\rho^1(z)\rt)}{\sin\lt(\rho^1(z)\rt)}<\gamma$ for some $\gamma=\gamma(n, k, \p\Omega_1, \p\Omega_0)>0$ sufficiently small, there exist some positive constants $C_1, T_1$ such that $\lbg<1$ on $\p\Omega_0,$
and $\s_j(\bn^2 \lbg)>0$ in $\mathcal M$ for all $1\leq j\leq k.$ Moreover, $\lbg$ satisfies
 $\bn_{\nu}\lbg>0$ on $\p\Omega_1,$ where $\nu$ is the unit normal to $\p\Omega_1$ that points away from the origin.\\
\end{lemma}
\begin{proof}
Since the proofs of part (1), (2), and (3) are similar and part (1) is the simplest, we only present the proofs of part (2) and (3).

In this proof we denote $A:=\frac{T_1e^{T_1\fbo}}{C_1}$ and $B:=\frac{T_1^2e^{T_1\fbo}}{C_1}=T_1A.$

Proof of (2): When $\N^n=\mH^n,$ at any point $x\in\mathcal M,$ under the specific coordinates we chose in Section \ref{sec-rdf},
$$\bn^2\lbg=M+A\fbo I_n,$$
where $M$ is given in \eqref{hessian-psi}. It is easy to see that, there exists $T_1=T_1(\p\Omega_0, \p\Omega_1)>0$ such that
$\s_1(M)>0$.
For $2\leq j\leq k,$ since
\[\s_j(a_{pq})=a_{11}\s_{j-1}(a_{\al\beta})+\s_j(a_{\al\beta})-\sum_{s=2}^{n-1}a_{1s}^2\s_{j-2}(a_{\al\beta}|s)\]
and $\la(a_{pq})\in\Sigma_k,$
we have for $2\leq j\leq k,$
$$\s_j^{11}=\s_{j-1}(a_{\al\beta})>0,$$
which implies $\la(a_{\al\beta})\in\Sigma_{k-1}.$
Moreover,
$$\s_j^{11}=\s_{j-1}(a_{\al\beta})<\sum_{i=1}^{n-1}\s_{j}^{ii}=(n-j)\s_{j-1}(a_{pq}).$$
By virtue of \eqref{sjmh} we can see that when $2\leq j\leq \min(n-1, k),$
\[
\begin{aligned}
\s_j(M)&>\lt(\frac{Aw}{\sinh r}\rt)^{j-1}B\lt[\frac{\cosh^2(r-\rho^1)}{\sinh^2(\rho^1)}\s_{j-1}(a_{pq})w^2
-\frac{\sinh^2(r-\rho^1)}{\sinh^2(\rho^1)}\s_{j-1}(a_{\al\beta})(w^2-1)\rt]-CA^j\\
&>\lt(\frac{Aw}{\sinh r}\rt)^{j-1}B\lt[\frac{\cosh^2(r-\rho^1)-(n-j)\sinh^2(r-\rho^1)}{\sinh^2(\rho^1)}\s_{j-1}(a_{pq})w^2\rt]
-CA^j.
\end{aligned}
\]
When $j=k=n$ we have
\[\s_n(M)=w^2\lt(\frac{Aw}{\sinh r}\rt)^{n-1}B\frac{\cosh^2(r-\rho^1)}{\sinh^2\rho^1}\s_{n-1}(a_{pq}).\]
Therefore, when
\[\max_{\mS^{n-1}}(\rho^0-\rho^1)<\tanh^{-1}\lt(\frac{1}{\sqrt{n-2}}\rt)\] there exists
$T_1=T_1(\p\Omega_1, \p\Omega_0)>0$ such that $\s_j(M)>0$ for all $1\leq j\leq k,$ that is, $\la(M)\in\Sigma_k.$

Now, we fix this $T_1$, then choose $C_1=C_1(T_1, \p\Omega_1, \p\Omega_0)>0$ such that $\lbg<0$ on $\p\Omega_0.$ It is clear that $\lbg<0$ in $\mathcal M.$ Therefore, for all $1\leq j\leq k$ we have
\[\s_j(\bn^2\lbg-\lbg g^\mH)\geq \s_j(M+A\fbo I_n)>\s_j(M)>0,\]
and $\lbg$ is the desired auxiliary function described in part (2).

Proof of (3): When $N=\mS^{n},$ at any point $x\in\mathcal M,$ under the specific coordinates we chose in Section \ref{sec-rdf},
$$\bn^2\lbg=M-A\fbo I_n,$$
where $M$ is given in \eqref{hessian-psi}. It is easy to see that, there exists $T_1=T_1(n, \p\Omega_0, \p\Omega_1)>0$ such that
$\s_1(M-A\fbo I_n)>0$. For all $2\leq j\leq \min(n-1, k),$ in view of Lemma \ref{matrix-lem} and the assumption that $\la(a_{pq})\in\Sigma_k,$ we obtain
\[
\begin{aligned}
\s_j(M-A\fbo I_n)&=\sum_{l=0}^jC^{j-l}_{n-l}(-A\fbo)^{j-l}\s_l(M)\\
&=\s_j(M)+\sum_{l=0}^{j-1}C^{j-l}_{n-l}(-A\fbo)^{j-l}\s_l(M)\\
&>c_jA^{j-1}B-\sum_{l=1}^{j-1}c_lA^{j-1}(\fbo)^{j-l}B-c_0A^j,
\end{aligned}
\]
where $c_i=c_i(n, i, \p\Omega_1, \p\Omega_0)>0,\,\,\mbox{for}\,\,0\leq i\leq j.$
When $j=k=n$ we have
\[\s_n(M)=w^2\lt(\frac{Aw}{\sin r}\rt)^{n-1}B\frac{\cos^2(r-\rho^1)}{\sin^2\rho^1}\s_{n-1}(a_{pq}).\]
Now, let
$\gamma=\gamma(c_0, \cdots, c_j)>0$ so small that $c_j-\sum_{l=1}^{j-1}c_l\gamma^{j-l}>0.$ Then when $\max\limits_{z\in\mS^{n-1}}\frac{\sin\lt(\rho^0(z)-\rho^1(z)\rt)}{\sin\lt(\rho^1(z)\rt)}<\gamma,$ there exists $T_1=T_1(\p\Omega_1, \p\Omega_0)>0$ such that $\s_j(M-A\fbo)>0$ for all $1\leq j\leq k.$

Now, we fix this $T_1$, then choose $C_1=C_1(T_1, \p\Omega_1, \p\Omega_0)>0$ such that $\lbg<1$ on $\p\Omega_0.$ It is clear that
\[\s_j(\bn^2\lbg)=\s_j(M-A\fbo I_n)>0\,\,\mbox{for all}\,\,1\leq j\leq k.\]
Therefore, $\lbg$ is the desired auxiliary function described in part (3).
\end{proof}

\begin{lemma}
\label{con-sub-out}
Let $\mathcal M=\Omega_0\setminus\bar\Omega_1,$ $\p\Omega_i=\{(z, \rho^i(z)): z\in\mS^{n-1}\}\subset\N^n(K)$ for $i=0, 1,$
and $\rho^0(z)>\rho^1(z)$ for all $z\in\mS^{n-1}.$ Set
$$\mathbf{\ubar{f}}:=C_2\lt(e^{T_2\fbz}-1\rt)+\frac{K+1}{2},\,\,\mbox{for}\,\,\fbz:=\frac{\phi(r-\rho^0)}{\phi(\rho^0)},$$
we have the following conclusions.\\

 (1). When $\N^n=\R^n,$ $\p\Omega_0$ is strictly $(k-1)$-convex, then there exist some positive constants $C_2, T_2$ such that $\lbf<-1/2$ on $\p\Omega_1,$
and $\s_j(\bn^2 \lbf)>0$ in $\mathcal M$ for all $1\leq j\leq k.$ Moreover, $\lbf$ satisfies
 $\bn_{\nu}\lbf>0$ on $\p\Omega_0,$ where $\nu$ is the unit normal to $\p\Omega_0$ that points away from the origin.\\

 (2). When $\N^n=\mH^n,$ $\p\Omega_0$ is strictly $k$-convex and $\max\limits_{z\in\mS^{n-1}}(\rho^0(z)-\rho^1(z))<\tanh^{-1}\lt(\frac{1}{\sqrt{n-2}}\rt)$, there exist some positive constants $C_2, T_2$ such that $\lbf<-1$ on $\p\Omega_1,$
and $\s_j(\bn^2 \lbf-\lbf g^{\mH})>0$ in $\mathcal M$ for all $1\leq j\leq k.$ Moreover, $\lbf$ satisfies
 $\bn_{\nu}\lbf>0$ on $\p\Omega_0,$ where $\nu$ is the unit normal to $\p\Omega_0$ that points away from the origin.\\

 (3). When $\N^n=\mS^n,$ $\p\Omega_0$ is strictly $k$-convex, and $\mathcal M\ssubset\mS^n_+,$ there exist some positive constants $C_2, T_2$ such that $\lbf<0$ on $\p\Omega_1,$
and $\s_j(\bn^2 \lbf)>0$ in $\mathcal M$ for all $1\leq j\leq k.$ Moreover, $\lbf$ satisfies
 $\bn_{\nu}\lbf>0$ on $\p\Omega_0,$ where $\nu$ is the unit normal to $\p\Omega_0$ that points away from the origin.\\
\end{lemma}
\begin{proof}
Since the proof of this lemma is very similar to the proof of Lemma \ref{con-sub-in}, we only provide a sketch for part (2) and (3) here.

In this proof, we denote $A:=C_2T_2e^{T_2\fbz}$ and $B:=T_2A.$ We also note that $\fbz<0$ in $\mathcal M.$

Proof of (2): When $\N^n=\mH^n,$ at any $x\in\mathcal M,$ under the specific coordinates we chose in Section \ref{sec-rdf},
$$\bn^2\lbf-\lbf g^\mH=M+(A\fbz-\lbf)I_n.$$ Note that
\[
\begin{aligned}
A\fbz-\lbf&=A\fbz-\frac{A}{T_2}+C_2\\
&=-\frac{A}{T_2}+C_2\lt(1+e^{T_2\fbz}T_2\fbz\rt)\\
&\geq-\frac{A}{T_2}+C_2(1-e^{-1})>-\frac{A}{T_2},
\end{aligned}
\]
where we have used $e^xx\geq-e^{-1}$ for $x\in\mathbb R.$
Therefore, we have $$\s_j(M+A\fbz I_n-\lbf I_n)>\s_j\lt(M-\frac{A}{T_2}I_n\rt).$$
Applying Lemma \ref{matrix-lem} and following the calculation in Lemma \ref{con-sub-in} we obtain,
when $\max\limits_{\mathbb S^{n-1}}\lt(\rho^0-\rho^1\rt)<\tanh^{-1}\lt(\frac{1}{\sqrt{n-2}}\rt)$ there exist $T_2=T_2(\p\Omega^0, \p\Omega^1)>0$ such that
\[
\begin{aligned}
&\s_j\lt(M-\frac{A}{T_2} I_n\rt)\\
>&c_jA^{j-1}B-\sum\limits_{l=1}^{j-1}c_l\lt(\frac{A}{T_2}\rt)^{j-l}A^{l-1}B-c_0A^j,
\end{aligned}
\]
where $c_i=c_i(n, i, \p\Omega_0, \p\Omega_1)>0,$ for $0\leq i\leq j.$ The remainder of the proof is exactly the same as that of Lemma \ref{con-sub-in}, so we omit it.

Proof of (3):  When $N=\mS^{n},$ at any $x\in\mathcal M,$ under the specific coordinates we chose in Section \ref{sec-rdf},
$$\bn^2\lbf=M-A\fb^0 I_n.$$
Since $-A\fb^0>0$ in $\mathcal M,$ following the calculation in Lemma \ref{con-sub-in} we obtain
\[\s_j(\bn^2\lbf)>\s_j\lt(M\rt)>c_jA^{j-1}B-c_0A^{j},\]
where $c_i=c_i(\p\Omega_0, \p\Omega_1)>0,$ for $i=0, j.$ This implies that
there exists $T_2=T_2(\p\Omega_1, \p\Omega_0)>0$ such that $\s_j(\bn^2\lbf)>\s_j\lt(M\rt)>0$ for all $1\leq j\leq k.$
The remainder of the proof is exactly the same as that of Lemma \ref{con-sub-in}, so we omit it.
\end{proof}

We now construct a subsolution $\lu$ described in Proposition \ref{app-prop1}.
\begin{proof}[Proof of Proposition \ref{app-prop1}.] We apply Lemma \ref{glue-lem} and set
$$\lu(x):=\frac{\lbf(x)+\lbg(x)}{2}+\frac{h(\lbf(x)-\lbg(x))}{2},$$
where $h$ is the function as in Lemma \ref{glue-lem}, $\lbg$ is constructed in Lemma \ref{con-sub-in}, and $\lbf$ is constructed in Lemma \ref{con-sub-out}.
It is easy to see that $\lu\geq\max(\lbf, \lbg)$ and
\[
\lu(x)=
\left\{\begin{aligned}
&\lbf(x),\,\,\mbox{if}\,\,\lbf(x)-\lbg(x)>\delta\\
&\lbg(x),\,\,\mbox{if}\,\,\lbg(x)-\lbf(x)>\delta.
\end{aligned}\right.
\]
We only need to verify $\lu$ is the desired subsolution of \eqref{app-main} on the set $\mathcal C:=\{x\in\mathcal M: \lt|\lbf(x)-\lbg(x)\rt|\leq \delta\}.$

As in the proof of Lemma 3.2 in \cite{Guan02}, when $\lt|\lbf(x)-\lbg(x)\rt|\leq\delta$ we get
\be\label{con-1}
\bn_{ij}\lu\geq\frac{1+t(x)}{2}\bn_{ij}\lbf+\frac{1-t(x)}{2}\bn_{ij}\lbg
\ee
for $t(x):=h'\lt(\lbf(x)-\lbg(x)\rt)\in[-1, 1].$
Moreover, in view of the concavity of $F$, for any symmetric matrices $A, B$ that satisfy $\la(A), \la(B)\in\Sigma_k$ we have,
\be\label{con-2}F\lt((1-\al)A+\al B\rt)\geq(1-\al)F(A)+\al F(B)\,\,\mbox{for all}\,\,\al\in[0, 1].\ee
Thus, we arrive at following conclusions.

In $\R^n,$ when $\lt|\lbf-\lbg\rt|\leq\delta,$
\[F\lt(\bn_{ij}\lu\rt)\geq\frac{1+t(x)}{2}F\lt(\bn_{ij}\lbf\rt)+\frac{1-t(x)}{2}F\lt(\bn_{ij}\lbg\rt)>0.\]

In $\mH^n,$ by virtue of properties (1), (2) in Lemma \ref{glue-lem}, when $\lt|\lbf-\lbg\rt|\leq\delta$ we have,
\[
\begin{aligned}
\lu&\leq\frac{\lbf+\lbg}{2}+\frac{h(0)}{2}+\frac{t(x)}{2}\lt(\lbf-\lbg\rt)\\
&\leq \frac{1+t(x)}{2}\lbf(x)+\frac{1-t(x)}{2}\lbg(x)+\frac{\delta}{2}.
\end{aligned}
\]
Combining with \eqref{con-1} gives
\[
\bn_{ij}\lu-\lu\delta_{ij}
\geq\frac{1+t(x)}{2}\lt(\bn_{ij}\lbf-\lbf\delta_{ij}\rt)+\frac{1-t(x)}{2}\lt(\bn_{ij}\lbg-\lbg\delta_{ij}\rt)-\frac{\delta}{2}\delta_{ij}.
\]
When $\delta>0$ is sufficiently small, in view of \eqref{con-2} we have
\[F\lt(\bn_{ij}\lu-\lu g^\mH\rt)>0.\]

In $\mS^n,$ by virtue of $\lu\geq 0,$ when $\lt|\lbf-\lbg\rt|\leq\delta$ we have,
\[
\bn_{ij}\lu+\lu\delta_{ij}\geq\bn_{ij}\lu\geq\frac{1+t(x)}{2}\bn_{ij}\lbf+\frac{1-t(x)}{2}\bn_{ij}\lbg.
\]
Same as the case when $\N^n=\R^n$ we have
\[F\lt(\bn_{ij}\lu+\lu g^\mS\rt)\geq F(\bn_{ij}\lu)>0.\]

Therefore, $\lu$ is the desired subsolution described in Proposition \ref{app-prop1}.
\end{proof}

\subsection{Construction of local subsolutions}\label{subsec-loc} In this subsection, we will construct a local subsolution $\lul$ which has been used in the proof of $C^2$ estimates.
Let $\e>0$ be a small constant, denote
\[\Omega_1^\e:=\{x\in\mathcal M: \text{dist}(x, \p\Omega_1)<\e\}.\]
We may assume $\e$ is so small that $\lu=\lbg$ in $\Omega_1^\e.$
We also denote
\[a^\e_0:=\min\limits_{\Omega_1^\e}\frac{\p\lu}{\p\fbo}\,\,\mbox{and}\,\,\psi^\e_K:=\max\limits_{\Omega^\e_1}F(\bn_{ij}\lu+K\lu g^\N).\]
Now consider
\[\lul:=\frac{1}{2}a^\e_0\fbo+\frac{T}{2}a^\e_0(\fbo)^2+\frac{K-1}{2}.\]
Similar to the argument in Lemma \ref{con-sub-in}, we can show that, when $T>0$ is sufficiently large, we have
\[F(\bn_{ij}\lul+K\lul g^\N)>c_ka_0^\e T^{\frac{1}{k}}.\]
Choosing $T\geq\lt(\frac{2\psi^\e_K}{c_ka_0^\e}\rt)^k$ we have,
$$F(\bn_{ij}\lul+K\lul g^\N)>2\psi^\e_K\,\,\mbox{in}\,\,\Omega_1^\e.$$
Now, let $\td\e<\e$ be a small positive constant such that
\[\frac{a^\e_0}{2}+Ta_0^\e\fbo<a_0^\e\,\,\mbox{in}\,\,\Omega_1^{\td\e}.\]
Then, it is easy to see that $\lu>\lul$ in $\Omega^{\td\e}_1.$ This completes the proof of the \textbf{Claim} we made in Section \ref{app}.

\subsection{Construction of upper barriers} In this subsection, we will construct an upper barrier $\bar u$ of the admissible solution $u$ of \eqref{app-main} described in Proposition \ref{app-prop2}.
We note that, for any admissible solution of \eqref{app-main}, we have $\s_1\lt(\bn^2 u+Kug^\N\rt)\geq0.$ This implies our upper barrier only needs to satisfy
$\s_1\lt(\bn^2 u+Kug^\N\rt)\leq0.$ Therefore, to construct an upper barrier, no curvature condition on the boundary is required.
\begin{lemma}
\label{con-super-in}
Let $\mathcal M=\Omega_0\setminus\bar\Omega_1,$ $\p\Omega_i=\{(z, \rho^i(z)): z\in\mS^{n-1}\}\subset\N^n(K)$ for $i=0, 1,$
and $\rho^0(z)>\rho^1(z)$ for all $z\in\mS^{n-1}.$ Set
$$\bbg:=C_3\lt(1-e^{-T_3\fbo}\rt)+\frac{K-1}{2},\,\,\mbox{for}\,\,\fbo:=\frac{\phi(r-\rho^1)}{\phi(\rho^1)},$$
we have the following conclusions.\\

 (1). When $\N^n=\R^n,$ then there exist some positive constants $C_3, T_3,$ such that $\bbg>1/2$ on $\p\Omega_0,$
and $\s_1(\bn^2 \bbg)<0$ in $\mathcal M.$\\

 (2). When $\N^n=\mH^n,$ then there exist some positive constants $C_3, T_3,$ such that $\bbg>0$ on $\p\Omega_0,$
and $\s_1(\bn^2 \bbg-\bbg g^{\mH})<0$ in $\mathcal M.$\\

 (3). When $\N^n=\mS^n,$ $\mathcal M\ssubset\mS^n_+,$ and
 $\max\limits_{z\in\mS^{n-1}}\frac{\sin\lt(\rho^0(z)-\rho^1(z)\rt)}{\sin\lt(\rho^1(z)\rt)}<\gamma$ for some $\gamma=\gamma(\p\Omega_1, \p\Omega_0)>0$ sufficiently small, there exist some positive constants $C_3, T_3,$ such that $\bbg>1$ on $\p\Omega_0,$
and $\s_1(\bn^2 \bbg+\bbg g^{\mS})<0$ in $\mathcal M.$\\
\end{lemma}
\begin{proof}
Since the proofs of parts (1), (2), and (3) are similar, we only present the proof of part (3).

In this proof, we denote $A:=C_3T_3e^{-T_3\fbo}$ and $B=-T_3A.$
When $N=\mS^{n},$ at any point $x\in\mathcal M,$ under the specific coordinates we chose in Section \ref{sec-rdf},
$$\bn^2\bbg+\bbg g^\mS=M-A\fbo I_n+\bbg I_n,$$
where $M$ is given in \eqref{hessian-psi}.
In view of \eqref{s1m} we have
\[
\begin{aligned}
\s_1(\bn^2\bbg+\bbg g^\mS)&<c_0A-\frac{T_3A\cos^2(r-\rho^1)}{\sin^2\rho^1}+n\lt[C_3\lt(1-e^{-T_3\fbo}\rt)\rt]\\
&<c_0A-c_1T_3A+\frac{nAe^{T_3\fbo}}{T_3}-\frac{nA}{T_3},
\end{aligned}
\]
where $c_i=c_i(\p\Omega_1, \p\Omega_0)$ for $i=0, 1.$
Therefore, there exists $T_3=T_3(\p\Omega_0, \p\Omega_1, n)>0$ such that
\[c_0A-c_1T_3A-\frac{nA}{T_3}<-\frac{3nA}{T_3}.\]
For this fixed $T_3,$ let
$\gamma=\frac{1}{T_3}.$ A straightforward calculation yields, when $\max\limits_{z\in\mS^{n-1}}\frac{\sin\lt(\rho^0(z)-\rho^1(z)\rt)}{\sin\lt(\rho^1(z)\rt)}<\gamma,$ we have
\[\s_1(\bn^2\bbg+\bbg g^\mS)<A\lt(\frac{-3n}{T_3}+\frac{ne}{T_3}\rt)<0.\]
Now, we choose $C_3=C_3(T_3, \p\Omega_1, \p\Omega_0)>0$ such that $\bbg>1$ on $\p\Omega_0.$
Then, $\bbg$ is the desired  upper barrier of the admissible solution $u$ of \eqref{app-main}.
\end{proof}
Similarly, we can prove the following Lemma.
\begin{lemma}
\label{con-super-out}
Let $\mathcal M=\Omega_0\setminus\bar\Omega_1,$ $\p\Omega_i=\{(z, \rho^i(z)): z\in\mS^{n-1}\}\subset\N^n(K)$ for $i=0, 1,$
and $\rho^0(z)>\rho^1(z)$ for all $z\in\mS^{n-1}.$ Set
$$\bbf:=\frac{1}{C_4}\lt(1-e^{-T_4\fbz}\rt)+\frac{K+1}{2},\,\,\mbox{for}\,\,\fbz:=\frac{\phi(r-\rho^0)}{\phi(\rho^0)},$$
we have the following conclusions.\\

 (1). When $\N^n=\R^n,$ then there exist some positive constants $C_4, T_4,$ such that $\bbf>-1/2$ on $\p\Omega_1,$
and $\s_1(\bn^2 \bbf)<0$ in $\mathcal M.$\\

 (2). When $\N^n=\mH^n,$ then there exist some positive constants $C_4, T_4,$ such that $\bbf>-1$ on $\p\Omega_1,$
and $\s_1(\bn^2 \bbf-\bbf g^{\mH})<0$ in $\mathcal M.$\\

 (3). When $\N^n=\mS^n,$ $\mathcal M\ssubset\mS^n_+,$ and
 $\max\limits_{z\in\mS^{n-1}}\frac{\sin\lt(\rho^0(z)-\rho^1(z)\rt)}{\sin\lt(\rho^0(z)\rt)}<\gamma$ for some $\gamma=\gamma(\p\Omega_1, \p\Omega_0)>0$ sufficiently small, there exist some positive constants $C_4, T_4,$ such that $\bbf>0$ on $\p\Omega_1,$
and $\s_1(\bn^2 \bbf+\bbf g^{\mS})<0$ in $\mathcal M.$\\
\end{lemma}

We now construct an upper barrier $\bar u$ described in Proposition \ref{app-prop2}.
\begin{proof}[Proof of Proposition \ref{app-prop2}.] We apply Lemma \ref{glue-lem} and set
$$\bar u(x):=\frac{\bbf(x)+\bbg(x)}{2}-\frac{h(\bbf(x)-\bbg(x))}{2},$$
where $h$ is the function as in Lemma \ref{glue-lem}, $\bbg$ is constructed in Lemma \ref{con-super-in}, and $\bbf$ is constructed in Lemma \ref{con-super-out}.
It is easy to see that $\bar u\leq\min(\bbf, \bbg)$ and
\[
\bar u(x)=
\left\{\begin{aligned}
&\bbg(x),\,\,\mbox{if}\,\,\bbf(x)-\bbg(x)>\delta\\
&\bbf(x),\,\,\mbox{if}\,\,\bbg(x)-\bbf(x)>\delta.
\end{aligned}\right.
\]
We only need to verify that $\bar u$ satisfies $\s_1(\bn^2\bar u+K\bar ug^\N)<0$  on the set $\mathcal C:=\{x\in\mathcal M: \lt|\bbf(x)-\bbg(x)\rt|\leq \delta\}.$
By the convexity of $h$ we have
\[\bar{\Laplace}\bar u\leq\frac{1-t(x)}{2}\bar{\Laplace}\bbf+\frac{1+t(x)}{2}\bar{\Laplace}\bbg,\]
where $t(x)=h'(\bbf(x)-\bbg(x)).$ Therefore,
\[\begin{aligned}
\bar\Laplace\bar u+nK\bar u&\leq\frac{1-t(x)}{2}\bar{\Laplace}\bbf+\frac{1+t(x)}{2}\bar{\Laplace}\bbg
+nK\lt(\frac{\bbf+\bbg}{2}-\frac{h(\bbf-\bbg)}{2}\rt)\\
&=\frac{1-t(x)}{2}\lt(\Laplace\bbf+nK\bbf\rt)+\frac{1+t(x)}{2}\lt(\Laplace\bbg+nK\bbg\rt)
+nK\lt(\frac{t(x)}{2}(\bbf-\bbg)-\frac{h(\bbf-\bbg)}{2}\rt)\\
&\leq\frac{1-t(x)}{2}\lt(\Laplace\bbf+nK\bbf\rt)+\frac{1+t(x)}{2}\lt(\Laplace\bbg+nK\bbg\rt)+\frac{n\delta}{2}.
\end{aligned}\]

It is easy to see that when $\delta>0$ is chosen small enough,
$\bar u$ is the desired  upper barrier of the admissible solution $u$ of \eqref{app-main}.
\end{proof}

\section*{Acknowledgements}
This work was completed while L.X. was visiting the Cornell Mathematics Department with support from the Ruth I. Michler Memorial Prize. L.X. is very grateful to the Michler family, the Association for Women in Mathematics, and the Cornell Mathematics Department.

\end{document}